\documentclass[reqno]{amsart}

\usepackage[utf8]{inputenc}
\usepackage[T1]{fontenc}
\usepackage{amsthm}
\usepackage{amsmath}
\usepackage{amssymb}
\usepackage{array}
\usepackage[english]{babel}
\usepackage{bm}
\usepackage{enumitem}
\usepackage[dvips]{graphicx}
\usepackage{color}
\usepackage{hyperref}
\usepackage{url}
\usepackage{stackrel}
\usepackage[left=3cm, right=3cm, paperheight=12.8in]{geometry}
\usepackage{fancyhdr}
\usepackage{mathrsfs}
\usepackage{relsize}
\usepackage{stmaryrd}
\usepackage{soul}
\usepackage{nicefrac}
\usepackage{datetime}
\usepackage{mathtools}
\longdate

\newtheorem{theorem}{Theorem}
\newtheorem{lemma}{Lemma}
\newtheorem{corollary}{Corollary}

\theoremstyle{definition}

\newtheorem*{def*}{Definition}

\newtheorem{remark}{\textbf{Remark}}

\theoremstyle{remark}

\newtheorem*{claim*}{\textsc{Claim}}

\hypersetup{
    pdftitle={Finite Partially Exchangeable Laws are Signed Mixtures of Product Laws},
    pdfmenubar=false,
    pdffitwindow=true,
    pdfstartview=FitH,
    colorlinks=true,
    linkcolor=blue,
    citecolor=green,
    urlcolor=cyan
}

\DeclareMathSymbol{\widehatsym}{\mathord}{largesymbols}{"62}

\renewcommand{\rho}{\varrho}

\providecommand{\NNb}{\mathbf{N}}

\renewcommand{\xi}{a}

\begin{document}
\title{Finite Partially Exchangeable Laws are Signed \\ Mixtures of Product Laws}

\author{Paolo Leonetti}
\address{Universit\`a ``Luigi Bocconi'' -- via Roentgen 1, 20136 Milano, Italy}
\email{leonetti.paolo@gmail.com}
\urladdr{\url{https://sites.google.com/site/leonettipaolo/}} 

%


\begin{abstract}
\noindent{} Given a partition $\{I_1,\ldots,I_k\}$ of $\{1,\ldots,n\}$, let $(X_1,\ldots,X_n)$ be random vector with each $X_i$ taking values in an arbitrary measurable space $(S,\mathscr{S})$ such that their joint law is invariant under finite permutations of the indexes within each class $I_j$. Then, it is shown that this law has to be a signed mixture of independent laws and identically distributed within each class $I_j$. 

We provide a necessary condition for the existence of a nonnegative directing measure. This is related to the notions of infinite extendibility and reinforcement. In particular, given a finite exchangeable sequence of Bernoulli random variables, the directing measure can be chosen nonnegative if and only if two effectively computable matrices are positive semi-definite.
\end{abstract}
\maketitle
\thispagestyle{empty}

\section{Introduction}\label{sec:introduction}

de Finetti's theorem, in one of its most general forms, states that if $(X_n: n \in \mathbf{N})$ is a sequence of random variables such that each $X_n$ takes values in a Borel space $(S,\mathscr{S})$ and their joint law is invariant under finite permutations of the indexes, then there exists a unique probability measure $\mu$ on the set $\mathcal{P}(S)$ of probability measures on $S$ such that 
\begin{equation}\label{eq:definetti}
\mathrm{Pr}\left((X_n: n \in \mathbf{N}) \in A\right)=\int_{\mathcal{P}(S)} \nu^\infty(A)\, \mu(\mathrm{d}\nu)
\end{equation}
for all $A \in \mathscr{S}^{\mathbf{N}}$. Here, $\nu^\infty$ stands for the (countably infinite) product measure $\nu \times \nu \times \cdots$, and $\mathcal{P}(S)$ is equipped with the $\sigma$-field generated by the weak${}^\star$ topology; see, for instance, \cite[Theorem 1.1]{Kall05}. More generally, the result holds if $S$ is a locally compact Hausdorff space and $\mathscr{S}$ its Baire $\sigma$-field, as shown by Hewitt and Savage in \cite[Theorem 7.2]{HSav}. On the other hand, some kind of assumptions on $S$ are needed to ensure representation \eqref{eq:definetti} holds: indeed, Dubins and Freedman have shown that there exists a separable metric space $S$ for which the result fails \cite[Theorem 2.14]{DubF}.

The finite case is completely different. Recently, Kerns and Sz\'ekely proved in \cite[Theorem 1.1]{KSze} that, if $(S,\mathscr{S})$ is an \emph{arbitrary} measurable space and $(X_1,\ldots,X_n)$ is exchangeable, then there exists a bounded \emph{signed} measure $\mu$ on the set $\mathcal{P}(S)$ of probability measures on $S$ such that 
\begin{equation}\label{eq:kernszek}
\mathrm{Pr}\left((X_1,\ldots,X_n) \in A\right)=\int_{\mathcal{P}(S)} \nu^{n}(A) \,\mu(\mathrm{d}\nu)
\end{equation}
for all $A \in \mathscr{S}^{n}$. Similarly, $\nu^n$ stands for the product measure $\nu \times \cdots \times \nu$, and $\mathcal{P}(S)$ is equipped with the smallest $\sigma$-field for which all the maps $\nu\mapsto \nu(B)$ are measurable, where $B$ ranges over $\mathscr{S}$. Accordingly, the directing measure $\mu$ cannot be assumed to be nonnegative, as it is shown in the well-known example provided by Diaconis and Freedman \cite{DF80}. An interpretation of the geometric structure underlying the proof of representation \eqref{eq:kernszek} can be found in \cite{Dia77}.

With these premises, the article focuses entirely on the finite case. In the first part, Theorem \ref{th:finiterepresentation} provides a generalization of the finite representation \eqref{eq:kernszek} to the case of finite partially exchangeable sequences, i.e., whenever the law of $(X_1,\ldots,X_n)$ is invariant under permutations of the indexes within each class $I_j$, for some partition $\{I_1,\ldots,I_k\}$ of $\{1,\ldots,n\}$. Under rather weak topological assumptions, this finite representation is unique if and only if the set of signed directing measures $\mu$ is compact. This type of invariance could be termed \emph{subgroup exchangeability}, cf. Remark \ref{lm:Gproduct}. Relationships with other types of partial exchangeability studied in the literature will be clarified in Section \ref{sec:mainresults}. As pointed out in \cite{KSze} and remarked in \cite{KonstYuan0}, the result still allows to prove the consistency of Bayesian estimators, namely, the sequence of posterior distributions relative to the true unknown parameter $\theta$ of the model converges to the degenerate distribution on $\theta$; cf. e.g. \cite[Proposition 4.1]{KSze} for the finite exchangeability case. See \cite{Barron} for an application 
to the Bayesian properties of normalized maximum likelihood.



In the second part, we provide a necessary condition to ensure the existence of a \emph{nonnegative} directing measure among all signed measures $\mu$ which satisfy the representation result, 
see Theorem \ref{th:truemixtures}.  
It turns out that the question is related to the notions of infinite extendibility and reinforcement. 
Lastly, we obtain necessary and sufficient conditions on exchangeable sequences taking values in $\{0,1\}$ to be mixtures, in the classical sense, of i.i.d. (Bernoulli) random variables, see Theorem \ref{th:positive}. In this regard, the problem can be equivalently reformulated to question whether a point belongs to the convex hull of a known finite set of extremal points; see, for instance, \cite[p. 275]{Dia77}. Here, differently from the geometric characterization, we show that an exchangeable law on $\{0,1\}^n$ is a mixture of i.i.d. random variables if and only if two effectively computable matrices are positive semi-definite.

\section{The finite representation}\label{sec:mainresults}

Given an index set $I$ and a subset $G$ of the group of permutations on $I$, a sequence $(X_i)_{i \in I}$ of random variables defined on a probability space $(\Omega, \mathscr{F}, \mathrm{P})$ with each $X_i$ taking values in a measurable space $(S, \mathscr{S})$ is said to be \emph{exchangeable over} $G$ whenever $(X_i: i \in I) \overset{d}{=} (X_{\sigma(i)}: i \in I)$ for all $\sigma \in G$.

Hereafter, let us suppose that $G$ can be written as the product of the symmetric groups on $I_1,\ldots,I_k$, for some partition $\{I_1,\ldots,I_k\}$ of $\{1,\ldots,n\}$. In other words, the joint law of $(X_1,\ldots,X_n)$ is invariant under permutations of the indexes within each class $I_j$. 

Then, our main result, which will be proved in Section \ref{sec:proofs}, can be stated as follows:

\begin{theorem}\label{th:finiterepresentation}
Let $(X_1,\ldots,X_n)$ be an exchangeable sequence over $G$ of random variables with each $X_i$ taking values in a measurable space $(S,\mathscr{S})$. Then there exists a bounded signed measure $\mu$ on $\mathcal{P}(S)^k$ such that
\begin{equation}\label{eq:representationfiniteexchangeable}
\mathrm{P}\left((X_i: i \in I_j) \in A_j, j=1,\ldots,k\right)=\int_{\mathcal{P}(S)^k} \nu_1^{n_1}(A_1) \cdots \nu_k^{n_k}(A_k) \,\mu(\mathrm{d}\nu_1,\ldots,\mathrm{d}\nu_k)
\end{equation}
for all $A_1 \in \mathscr{S}^{n_1},\ldots,A_k \in \mathscr{S}^{n_k}$, where $n_j$ stands for the cardinality of $I_j$ for each $j$.
\end{theorem}

As usual, here ``bounded'' means that if $\mu$ has Hahn decomposition $\mu^+-\mu^-$ then $\mu^++\mu^-$ is a finite measure. It is worth noting that Theorem \ref{th:finiterepresentation} generalizes the results of Jaynes \cite{Jaynes} and Kerns and Sz\'ekely \cite[Theorem 1.1]{KSze}. 

The main machinery of the proof can be tracked back to Janson, Konstantopoulos and Yuan \cite[Theorem 1]{KonstYuan0}. On the other hand, it solves an open question in the same article, see \cite[Section 4.8]{KonstYuan0}. However, our main interest in Theorem \ref{th:finiterepresentation} is related to its consequences. 
\begin{corollary}\label{cor:parametric2}
With the notation of Theorem \ref{th:finiterepresentation}, let us assume that $S$ is a separable Banach space with induced Borel $\sigma$-field $\mathscr{S}$. Then there exists a bounded signed measure $\eta$ on $S^k$ such that
\begin{equation}\label{eq:representationfiniteexchangeableparametric2}
\mathrm{P}\left((X_i: i \in I_j) \in A_j, j=1,\ldots,k\right)=\int_{S^k} \nu_{\theta_1}^{n_1}(A_1) \cdots \nu_{\theta_k}^{n_k}(A_k) \,\eta(\mathrm{d}\theta_1,\ldots,\mathrm{d}\theta_k)
\end{equation}
for all $A_1 \in \mathscr{S}^{n_1},\ldots,A_k \in \mathscr{S}^{n_k}$, where each $\nu_{\theta_j}$ is a probability measure on $S$ depending measurably on the parameter $\theta_j$.
\end{corollary}

Note that the result holds in the case of real-valued random variables. If the directing measure $\eta$ can be chosen nonnegative, this may provide a sort of justification for the use of priors in Bayesian parametric.

Lastly, the representation provided in Theorem \ref{th:finiterepresentation} gives, informally, the same amount of information when the assumption of exchangeability over an \emph{arbitrary} subset $G$ is replaced by the exchangeability over the largest \emph{subgroup} contained in $G$, cf. Remark \ref{lm:Gproduct}. This highlights the differences with other types of exchangeability considered in the literature. To make some examples, fix an array of random variables $X=(X_{i,j}: 1\le i\le n, 1\le j\le m)$. In the case of partial exchangeability \emph{\'{a} la de Finetti} \cite{deFinetti}), the joint law of $X$ is invariant under permutations of the random variables within each column: the information provided by our representation is maximal, since $G$ can be written exactly as product of symmetric groups on the indexes of the columns. On the other hand, if $X$ is partially exchangeable \emph{\'{a} la Aldous and Kallenberg}, cf. \cite{Ald81} and \cite{Kall89}, then its law is invariant under all pairs of [respectively, the same] permutations of the indexes of rows and columns; this is commonly known as separately row-column exchangeability [resp., jointly row-column exchangeability]. For instance, a $2\times 2$ array is separately row-column exchangeable if and only if, up to relabelings to indexes, it holds
$$
\begin{pmatrix}
X_1 & X_2 \\
X_3 &  X_4
 \end{pmatrix} \overset{d}{=}
\begin{pmatrix}
X_2 & X_1 \\
X_4 &  X_3
 \end{pmatrix} \overset{d}{=}
\begin{pmatrix}
X_3 & X_4 \\
X_1 &  X_2
 \end{pmatrix} \overset{d}{=}
\begin{pmatrix}
X_4 & X_3 \\
X_2 &  X_1
 \end{pmatrix}.
$$
This implies that the unique contained subgroup of permutations is the trivial one. In a sense, this minimizes the amount of information provided by our representation because it is as if there were no constraints. Similar considerations apply to other types of invariance considered in literature, e.g., Markov exchangeability introduced by Diaconis and Freedman \cite{DiaFre80}.

\section{True Mixtures}\label{sec:truemixtures}

In most cases, the main difference between the finite and the infinite case is that the directing measure $\mu$ may be signed. Therefore, it looks natural to ask about the positiveness of $\mu$ and, more in general, about the \emph{infinite extendibility} of an exchangeable sequence $(X_1,\ldots,X_n)$.

In this respect, Konstantopoulos and Yuan have shown in \cite[Theorem 2]{KonstYuan} that if $S$ is a locally compact Hausdorff space and the law of $X_1$ is inner and outer regular, then such an infinite extension exists if and only if for each integer $N\ge n$ there is an exchangeable sequence $(Y_1,\ldots,Y_N)$ with each $Y_i$ taking values in $(S,\mathscr{S})$ such that $(X_1,\ldots,X_n)\overset{d}{=} (Y_1,\ldots,Y_n)$. In turn, this is equivalent to be a \emph{true mixture} of i.i.d. random variables whenever $S$ is equipped with its Baire $\sigma$-field (hereafter, ``true'' underlines that the directing measure is nonnegative), see \cite[Theorem 3]{KonstYuan}. 

In this section, we provide a set of necessary conditions on the joint law of a random vector $(X_1,\ldots,X_n)$ exchangeable over $G$ which can be represented as true mixture of independent random variables and identically distributed within each class $I_j$. To this aim, we fix some additional notation: given a square matrix $\textsl{M}=(m_{i,j})$ with entries in $\mathscr{S}^{n_1} \times \cdots \times \mathscr{S}^{n_k}$, we shorten the real-valued matrix $\left(\mathrm{P}(X \in m_{i,j})\right)$ with $\mathrm{P}\,\textsl{M}$. Moreover, given positive integers $a,b$ and matrices $\textsl{A}, \textsl{B}$ with entries in $\mathscr{S}^a$ and $\mathscr{S}^b$, respectively, we denote by $\textsl{A} \square \textsl{B}$ the matrix with entries in $\mathscr{S}^{a+b}$, constructed as the analogue of Kronecker product, where the multiplication of entries is replaced by their cartesian products. 

Accordingly, we obtain the following result: 

\begin{theorem}\label{th:truemixtures}
Let $(X_1,\ldots,X_n)$ be an exchangeable sequence over $G$ of random variables with each $X_i$ taking values in a measurable space $(S,\mathscr{S})$, and let us suppose that the signed measure $\mu$ in the the finite representation \eqref{eq:representationfiniteexchangeable} is nonnegative.

Then, the following matrix
\renewcommand{\arraystretch}{1.2}
$$
\mathrm{P}\,\mathlarger{\square}_{j=1}^k
\begin{pmatrix}
  A_j^{2m_j}\times B_j & A_j^{2m_j-1}\times S \times B_j & \cdots & A_j^{m_j}\times S^{m_j} \times B_j \\
  A_j^{2m_j-1}\times S \times B_j & A_j^{2m_j-2}\times S^2 \times B_j & \cdots & A_j^{m_j-1}\times S^{m_j+1} \times B_j \\
  \vdots  & \vdots  & \ddots & \vdots  \\
  A_j^{m_j}\times S^{m_j} \times B_j & A_j^{m_j-1}\times S^{m_j+1} \times B_j & \cdots & S^{2m_j} \times B_j
 \end{pmatrix}
$$
\renewcommand{\arraystretch}{1}
is positive semi-definite for all even nonnegative integers $2m_1\le n_1$, $\ldots$, $2m_k \le n_k$, for all $A_1,\ldots,A_k \in \mathscr{S}$, and for all $B_1 \in \mathscr{S}^{n_1-2m_1},\ldots,B_k \in \mathscr{S}^{n_k-2m_k}$.
\end{theorem}

The main result of von Plato \cite[Theorem 4]{vonP} follows in the special case where $n=4$, $k=2$, $n_1=n_2=2$, $m_1=m_2=1$, $S=\{0,1\}$, and $A=\{1\}$; cf. Remark \ref{cor:CBS}. 


Moreover, it follows that, if $k=1$, $m_1=1$, and $\mathrm{P}(X \in S^2 \times B)>0$, then 
\begin{displaymath}
\begin{split}
\mathrm{P}(X_1 \in A| X_2 \in A, (X_3,\ldots,X_n) \in B)&=\frac{\mathrm{P}(X \in A^2\times B)}{\mathrm{P}(X \in A\times S \times B)} \\
&\ge \frac{\mathrm{P}(X \in A\times S\times B)}{\mathrm{P}(X \in S^2 \times B)}=\mathrm{P}(X_2 \in A| (X_3,\ldots,X_n) \in B).
\end{split}
\end{displaymath}
This can be intepreted as a \emph{reinforcement} property; see, e.g., \cite{MSW} and \cite{Pem}. The same observation applies to partially exchangeable sequences, once we set each $m_j=1$ in Theorem \ref{th:truemixtures}.


In the special case where $G$ is the symmetric group on $\{1,\ldots,n\}$ and $S=\{0,1\}$, we provide an explicit characterization of the finite exchangeable sequences which can be written as true mixtures of i.i.d. (Bernoulli) random variables. 

To this aim, define $x_i=\mathrm{P}(X_1=\cdots=X_i=0,X_{i+1}=\cdots=X_n=1)$ for each $i=0,1,\ldots,n$, and, for each positive integer $n$, the Hankel matrices $\textsl{H}_n=\left(h_{i,j}^{(n)}\right)$ and $\textsl{K}_n=\left(k_{i,j}^{(n)}\right)$ by
$$
h_{i,j}^{(n)}= \left\{
\begin{array}{*2{>{\displaystyle}l}p{5cm}}
\!\! \sum_{k=0}^{n+2-i-j}\binom{n+2-i-j}{k}x_k & \text{if } n\text{ is even and }1\le i,j\le \frac{n+2}{2} \\
\!\! \sum_{k=0}^{n+1-i-j}\binom{n+1-i-j}{k}x_k & \text{if } n\text{ is odd and }1\le i,j\le \frac{n+1}{2}
\end{array}
\!\!\right.,
$$
and
$$
k_{i,j}^{(n)}= \left\{
\begin{array}{*2{>{\displaystyle}l}p{5cm}}
\!\! \sum_{k=0}^{n-i-j}\binom{n-i-j}{k}x_{k+1} & \text{if } n\text{ is even and }1\le i,j\le \frac{n}{2} \\
\!\! \sum_{k=0}^{n+1-i-j}\binom{n+1-i-j}{k}x_{k+1} & \text{if } n\text{ is odd and }1\le i,j\le \frac{n+1}{2}
\end{array}
\!\!\right..
$$

\vspace{3mm}

With these premises, we obtain the following necessary and sufficient condition:

\begin{theorem}\label{th:positive}
Let $(X_1,\ldots,X_n)$ be an exchangeable sequence of $\{0,1\}$-valued random variables. Then the joint law of $(X_1,\ldots,X_n)$ is a true mixture of i.i.d. random variables if and only if $\textsl{H}_n$ and $\textsl{K}_n$ are positive semi-definite.
\end{theorem}
Interestingly, Wood \cite[Theorem 2]{Wood} calculated the probability that, picking at random an exchangeable sequence of $\{0,1\}$-valued random variables $(X_1,\ldots,X_n)$, this is infinitely extendible; in particular, this probability goes to $0$ as $n\to \infty$. 

Lastly, we obtain the following corollary, which turns out to be a generalization of \cite[Theorem 1 and 2]{vonP}:
\begin{corollary}\label{cor:reinfor}
Let $(X_1,\ldots,X_n)$ be a $\{0,1\}$-valued exchangeable sequence which can be represented as a true mixture of i.i.d. random variables. Then, for all $i=1,\ldots,n-1$, it holds
\begin{equation}\label{eq:reinf}
x_i \le \sqrt{x_{i-1}x_{i+1}}.
\end{equation}
On the other hand, these conditions are also sufficient if and only if $n \le 3$. 
\end{corollary}
An alternative proof of the first part of Corollary \ref{cor:reinfor} has been given in a manuscript of Muliere and Walker \cite{MulW}, which, however, focuses on the relationship between the concepts of reinforcement and finite exchangeability. Proofs of Theorem \ref{th:truemixtures} and Theorem \ref{th:positive} follow in Section \ref{sec:truemixtures}.

\section{Notations and Preliminaries}\label{sec:notations}

The sets of reals and positive integers are denoted, respectively, by $\mathbf{R}$ and $\NNb$. Each of these sets is endowed with its usual addition, multiplication, and (total) order $\le$.

Given sets $X,Y,Z$ and functions $f:X\to Y$ and $g:Y\to Z$, we write $f[X]$ for the image set (or range) of $f$, namely $f[X] := \{f(x): x \in X\} \subseteq Y$, and $g\circ f$ for the composition $X\to Z: x \mapsto g(f(x))$. Moreover, for each nonempty subset $A\subseteq X$, the indicator function $\bm{1}_A: X\to \{0,1\}$ is defined by $\bm{1}_A(x)=1$ if and only if $x \in A$.

Given a probability space $(X,\Sigma,\mu)$ and a measurable space $(Y,\mathscr{G})$, we say that two random variables $\alpha,\beta:X \to Y$ have the same distribution, shortened with $\alpha\overset{d}{=}\beta$, whenever $\mu \circ \alpha^{-1}=\mu \circ \beta^{-1}$. Also, for each $x \in X$, let $\delta_x$ be the Dirac measure at $x$, that is, the probability measure $\Sigma \to \mathbf{R}: A\mapsto \bm{1}_A(x)$. Lastly, the symbol $\otimes$ will be reserved for the Kronecker product. We refer to \cite{Eng}, \cite{Horn}, \cite{Kall02}, and \cite{RudinFA}, respectively, for basic aspects of topology, matrix analysis, probability theory, and functional analysis (including notation and terms not defined here).


\begin{remark}\label{lm:measurability}
Notice that the integral on the right hand side of \eqref{eq:representationfiniteexchangeable} is well defined. Indeed, for each $A_1 \in \mathscr{S}^{n_1},\ldots,A_k \in \mathscr{S}^{n_k}$, each map $(\nu_1,\ldots,\nu_k) \mapsto \nu_1^{n_1}(A_1)\cdots \nu_k^{n_k}(A_k)$ is measurable.

To this aim, it will be sufficient to show the statement for $k=1$. The $\sigma$-field defined on $\mathcal{P}(S)$, hereafter shortened with $\sigma(\mathcal{P}(S))$, may be written explicitly as $\sigma\left(\{\nu \in \mathcal{P}(S): \nu(A) \in B\}: A \in \mathscr{S}, B \in \mathscr{B}\right)$, where $\mathscr{B}$ stands for the usual Borel $\sigma$-field on $\mathbf{R}$. Then it is claimed that the map $\mathcal{P}(S)\to \mathbf{R}$ defined by $\nu \mapsto \nu^n(A)$ is measurable for each $n \in \mathbf{N}$ and $A \in \mathscr{S}^n$, that is, $\{\nu \in \mathcal{P}(S): \nu^n(A) \in B\}$ belongs to $\sigma(\mathcal{P}(S))$ for each $B \in \mathscr{B}$. Let $\mathscr{E}$ be the collection of all $A \in \mathscr{S}^n$ such that the mapping $\nu \mapsto \nu^n(A)$ is measurable. Then $\mathscr{E}$ contains all rectangles $A=A_1 \times \cdots \times A_n$ with $A_1,\ldots,A_n \in \mathscr{S}$, as far as the product of measurable functions $\nu \mapsto \nu(A_i)$ is measurable. In particular, $\mathscr{E}$ contains $S^n$. Moreover, it is easily seen that $\mathscr{E}$ is closed under finite intersection and countable disjoint union. Therefore, it follows by the monotone class theorem that $\mathscr{E}=\mathscr{S}^n$. 
\end{remark}

\begin{remark}\label{lm:Gproduct}
As anticipated in the Introduction, the type of partial exchangeability used in Theorem \ref{th:finiterepresentation} could be termed \emph{subgroup exchangeability}. 
Indeed, let $G$ be a subgroup of the symmetric group on $\{1,\ldots,n\}$. Then there exists a unique partition $\{I_1,\ldots,I_k\}$ of $\{1,\ldots,n\}$ such that $G$ is equal to the product of the symmetric groups on $I_1,\ldots,I_k$.

This is trivially true for $n=1$. Let us assume that it holds for $n-1$ and consider all permutations of $G$ which leaves the element $n$ unchanged. This subset, being isomorphic to a subgroup of the symmetric group on $\{1,\ldots,n-1\}$, can be written as the group the generated by the symmetric groups on $I_1,\ldots,I_k$, for a necessarily unique partition $\{I_1,\ldots,I_k\}$ of $\{1,\ldots,n-1\}$. If the subset is equal to $G$ itself, then the required partition equals $\{I_1,\ldots,I_k, \{n\}\}$; otherwise, let $J$ be the subset of integers in $\{1,\ldots,n-1\}$ which are connected with $n$, for some permutations in $G$. Then the required partition is a coarsening of $\{I_1,\ldots,I_k, \{n\}\}$, obtained making the union of $\{n\}$ with all $I_j$ containing at least one element of $J$. The claim follows by induction.
\end{remark}

\begin{remark}\label{rmk:Gtrivial}
Notice that if $G$ is the trivial subgroup of the symmetric group on $\{1,\ldots,n\}$, then \emph{all} joint laws of random vectors $(X_1,\ldots,X_n)$ are exchangeable over $G$. Hence, Theorem \ref{th:finiterepresentation} implies the existence of a bounded signed measure $\mu$ such that for all $A_1,\ldots,A_n \in \mathscr{S}$ it holds
$$
\mathrm{P}\left(X_1 \in A_1,\ldots,X_n \in A_n\right)=\int_{\mathcal{P}(S)^n} \nu_1(A_1) \cdots \nu_n(A_n) \,\mu(\mathrm{d}\nu_1,\ldots,\mathrm{d}\nu_n).
$$
This is not surprising. Indeed, joint laws can be always represented as \emph{true} mixtures of product laws. To prove this, denote by $\textsl{q}$ the measurable map $S^n \to \mathcal{P}(S)^n:(s_1,\ldots,s_n)\mapsto (\delta_{s_1},\ldots,\delta_{s_n})$. Then, it follows that, for all measurable sets $A_1,\ldots,A_n \in \mathscr{S}$, the probability $\mathrm{P}(X_1 \in A_1,\ldots,X_n \in A_n)$ is equal to
\begin{displaymath}
\begin{split}
\int_{S^n}\bm{1}_{A_1\times \cdots \times A_n}(s_1,\ldots,s_n)\mathrm{P}(\mathrm{d}s_1,\ldots,\mathrm{d}s_n)&=\int_{S^n}\delta_{s_1}(A_1) \cdots \delta_{s_n}(A_n)\mathrm{P}(\mathrm{d}s_1,\ldots,\mathrm{d}s_n) \\
&=\int_{\mathcal{P}(S)^n}\nu_1(A_1) \cdots \nu_n(A_n)(\mathrm{P}\circ \textsl{q}^{-1})(\mathrm{d}\nu_1,\ldots,\mathrm{d}\nu_n).
\end{split}
\end{displaymath}
\end{remark}

The next identity will be useful in the proof of Theorem \ref{th:positive}:
\begin{lemma}\label{lm:identity}
Let $n$ be a positive integer. Then, for each nonnegative integer $j\le n-1$ it holds
$$
\sum_{i=j}^{n-1}\binom{n}{i}\binom{i}{j}(-1)^i=\binom{n}{j}(-1)^{n-1}.
$$
\end{lemma}
\begin{proof}
We have equivalently to prove that $\sum_{i=j}^n\binom{n}i\binom{i}j(-1)^i=0$. To this aim, rewrite it as
$$
\sum_{i=j}^n\binom{n}i\binom{i}j(-1)^i=\binom{n}{j}\sum_{i=j}^n \binom{n-j}{i-j}(-1)^i=(-1)^j\binom{n}{j}\sum_{k=0}^{n-j}\binom{n-j}k(-1)^{k}.
$$
The claim follows by the fact that the last sum is the binomial expansion of $(1-1)^{n-j}$.
\end{proof}

We conclude the section with a result about positive semi-definite matrices:
\begin{lemma}\label{lm:CBS} 
Let $(X,\mathscr{F},\mu)$ be a finite (nonnegative) measure space and let $f_1,g_1,\ldots,f_k,g_k:X\to \mathbf{R}$ be nonnegative measurable functions. Fix also positive integers $n_1,\ldots,n_k$ such that $f_1^{i_1}g_1\cdots f_k^{i_k}g_k$ are integrable for all nonnegative integers $i_1\le 2n_1,\ldots,i_k\le 2n_k$, and define the matrix $\textsl{M}=(m_{i,j})$ by
\renewcommand{\arraystretch}{1.3}
$$\bigotimes_{j=1}^k
\begin{pmatrix}
   f_j^{2n_j}g_j & f_j^{2n_j-1}g_j & \cdots & f_j^{n_j}g_j \\
  f_j^{2n_j-1}g_j & f_j^{2n_j-2}g_j & \cdots & f_j^{n_j-1}g_j \\
  \vdots  & \vdots  & \ddots & \vdots  \\
  f_j^{n_j}g_j & f_j^{n_j-1}g_j & \cdots & g_j
 \end{pmatrix}.\renewcommand{\arraystretch}{1}
$$
Then, $\left(\int_X m_{i,j} \,\mathrm{d}\mu\right)$ is positive semi-definite. 
\end{lemma}
\begin{proof}
Let $\nu$ be the function $\mathscr{F} \to \mathbf{R}: F \mapsto \int_Fg_1\cdots g_k\,\mathrm{d}\mu$. Then, it is straightforward to check that $\nu$ is a finite (nonnegative) measure such that 
$$
\int_X h g_1\cdots g_k \,\mathrm{d}\mu=\int_Xh\,\mathrm{d}\nu
$$
whenever $h: X\to \mathbf{R}$ is measurable. It means that it is enough to prove the claim for $g_1=\cdots=g_k=\bm{1}_X$. In addition, the statement is trivial if $\mu(X)=0$, hence we can assume without loss of generality that $\mu$ is a probability measure on $X$.

At this point, for each $\bm{z} \in \mathbf{R}^{n_1\cdots n_k}$, we have to show that $\bm{z}^\prime \left(\int_X m_{i,j} \,\mathrm{d}\mu\right) \bm{z}$ is nonnegative, where $\bm{z}^\prime$ stands for the transpose of $\bm{z}$. Accordingly, define the measurable functions $h_{j,1},\ldots,h_{j,n_1\cdots n_k}:X \to \mathbf{R}$, for each $j=1,\ldots,k$, such that they have distribution $f_j$ and all $h_{j,i}$'s are jointly independent. Consider the following Kronecker product
\renewcommand{\arraystretch}{1.3}
$$\bigotimes_{j=1}^k
\begin{pmatrix}
   f_j^{2n_j} & f_j^{2n_j-1} & \cdots & f_j^{n_j} \\
  f_j^{2n_j-1} & f_j^{2n_j-2} & \cdots & f_j^{n_j-1} \\
  \vdots  & \vdots  & \ddots & \vdots  \\
  f_j^{n_j} & f_j^{n_j-1} & \cdots & 1
 \end{pmatrix}\renewcommand{\arraystretch}{1}
$$
and replace each function $f_j$ in the $i$-th column with a $h_{j,i}$, for all $i=1,\ldots,n_1\cdots n_k$. Averaging over all $(n_1\cdots n_k)!^k$ permutations of the columns, we obtain by a symmetric argument that the above matrix is equal to
\renewcommand{\arraystretch}{1.3}
$$\frac{1}{(n_1\cdots n_k)!^k} \bigotimes_{j=1}^k \sum_{i=1}^{n_1\cdots n_k}
\begin{pmatrix}
  h_{j,i}^{2n_j} & h_{j,i}^{2n_j-1} & \cdots & h_{j,i}^{n_j} \\
  h_{j,i}^{2n_j-1} & h_{j,i}^{2n_j-2} & \cdots & h_{j,i}^{n_j-1} \\
  \vdots  & \vdots  & \ddots & \vdots  \\
  h_{j,i}^{n_j} & h_{j,i}^{n_j-1} & \cdots & 1
 \end{pmatrix}.\renewcommand{\arraystretch}{1}
$$

Taking the expected values at each entry, and using independence, it follows that
\begin{displaymath}
\begin{split}
\bm{z}^\prime \left(\int_X m_{i,j} \,\mathrm{d}\mu\right) \bm{z}&=\frac{1}{(n_1\cdots n_k)!^k}\bm{z}^\prime\left(\int_X\bigotimes_{j=1}^k \sum_{i=1}^{n_1\cdots n_k}
\renewcommand{\arraystretch}{1.3}
\begin{pmatrix}
  h_{j,i}^{2n_j} & h_{j,i}^{2n_j-1} & \cdots & h_{j,i}^{n_j} \\
  h_{j,i}^{2n_j-1} & h_{j,i}^{2n_j-2} & \cdots & h_{j,i}^{n_j-1} \\
  \vdots  & \vdots  & \ddots & \vdots  \\
  h_{j,i}^{n_j} & h_{j,i}^{n_j-1} & \cdots & 1
	\end{pmatrix}\,\mathrm{d}\mu\right)\bm{z}\\
	&=\frac{1}{(n_1\cdots n_k)!^k}\int_X\bm{z}^\prime\bigotimes_{j=1}^k \sum_{i=1}^{n_1\cdots n_k}
\begin{pmatrix}
  h_{j,i}^{2n_j} & h_{j,i}^{2n_j-1} & \cdots & h_{j,i}^{n_j} \\
  h_{j,i}^{2n_j-1} & h_{j,i}^{2n_j-2} & \cdots & h_{j,i}^{n_j-1} \\
  \vdots  & \vdots  & \ddots & \vdots  \\
  h_{j,i}^{n_j} & h_{j,i}^{n_j-1} & \cdots & 1
	\end{pmatrix}\bm{z}\,\mathrm{d}\mu\\
	&=\frac{1}{(n_1\cdots n_k)!^k}\int_X\bm{z}^\prime\bigotimes_{j=1}^k \textsl{H}_j \textsl{H}_j^\prime\,\bm{z}\,\mathrm{d}\mu,
	\renewcommand{\arraystretch}{1}
\end{split}
\end{displaymath}
where we define
$$
\textsl{H}_j=
\renewcommand{\arraystretch}{1.3}
\begin{pmatrix}
  h_{j,1}^{n_j} & h_{j,2}^{n_j} & \cdots & h_{j,n_1\cdots n_k}^{n_j} \\
  h_{j,1}^{n_j-1} & h_{j,2}^{n_j-1} & \cdots & h_{j,n_1\cdots n_k}^{n_j-1} \\
  \vdots  & \vdots  & \ddots & \vdots  \\
  1 & 1 & \cdots & 1
	\end{pmatrix}
	\renewcommand{\arraystretch}{1}
$$
for each $j=1,\ldots,k$. This is clearly nonnegative since the Kronecker product of positive semi-definite matrices is positive semi-definite.
\end{proof}


\begin{remark}\label{cor:CBS}
It turns out that Lemma \ref{lm:CBS} is the determinant analogue of the Cauchy--Schwarz's inequality \cite[Chapter 2.6]{Mitri}, indeed: Let $(X,\mathscr{F},\mu)$ be a finite measure space and let $f,g:X\to \mathbf{R}$ be measurable nonnegative square-integrable functions. Then $fg$ is integrable and
$$
\left(\int_X fg\,\mathrm{d}\mu\right)^2 \le \left(\int_X f^2\,\mathrm{d}\mu\right)\left(\int_X g^2\,\mathrm{d}\mu\right).
$$

To this aim, set $k=2$, $n_1=n_2=1$ and $g_1=g_2=\bm{1}_X$ in Lemma \ref{lm:CBS} and assume without loss of generality that $\mu$ is a probability measure (the case $\mu(X)=0$ being trivial). Then, setting $f_1=f$ and $f_2=g$, we obtain that the matrix with entries given by the $\mu$-integrals of each entry in
$$
\begin{pmatrix}
  f^{2} & f \\
  f & 1
 \end{pmatrix} \mathlarger{\otimes}
\begin{pmatrix}
  g^{2} & g \\
  g & 1
 \end{pmatrix}
$$
is positive semi-definite. The claim follows by Sylvester's criterion \cite[Theorem 7.2.5]{Horn}, indeed the principal minor
$$
\renewcommand{\arraystretch}{1.3}
\begin{pmatrix}
  \int_X g^2\,\mathrm{d}\mu & \int_X fg\,\mathrm{d}\mu  \\
   \int_X fg\,\mathrm{d}\mu  & \int_X f^{2}\,\mathrm{d}\mu  
 \end{pmatrix}
\renewcommand{\arraystretch}{1}
$$
has a nonnegative determinant.
\end{remark}

\section{Proofs of Theorem \ref{th:finiterepresentation} and Corollary \ref{cor:parametric2}}\label{sec:proofs}


\begin{proof}[Proof of Theorem \ref{th:finiterepresentation}]
Denoting with $n_j$ the cardinality of $I_j$, we obtain by the exchangeability assumption that
$$
\mathrm{P}\left((X_i: i \in I_j) \in A_j, j=1,\ldots,k\right)=|G|^{-1}\sum_{\sigma \in G} \mathrm{P}\left(\sigma X \in A\right)=\int_\Omega \bm{U}_{X}(A)\,\mathrm{d}\mathrm{P}
$$
for each $A_1 \in \mathscr{S}^{n_1}, \ldots,A_k \in \mathscr{S}^{n_k}$, where $A=A_1\times \cdots A_k$, $\sigma X=(X_{\sigma(1)},\ldots,X_{\sigma(n)})$, and $\bm{U}_{X}(A):\Omega \to \mathcal{P}(S)$ is the mapping defined by the arithmetic mean of empirical distributions over permutations on $G$, i.e.,
$$
\omega \mapsto |G|^{-1}\sum_{\sigma \in G}\delta_{\sigma X(\omega)}(A).
$$

Let $\mathscr{T}$ be the $\sigma$-field generated by $\mathscr{S}$ and the singletons in $S$. Moreover, given $x=((x_i): i \in I_j, j=1,\ldots,k) \in S^n$, let $\nu_j(x)$ be the point measure on $S$ defined by $\sum_{i \in I_j}\delta_{x_i}$ for each $j=1,\ldots,k$. It follows by the exchangeability assumption that, for each $x \in S^n$ and $A \in \mathscr{S}^n$, it holds
$$
\bm{U}_x(A)=|\mathcal{S}(x)|^{-1}\sum_{y \in \mathcal{S}(x)}\delta_y(A),
$$
where $\mathcal{S}(x)$ stands for the set of all $y\in S^n$ such that $\nu_j(x)=\nu_j(y)$ for each $j=1,\ldots,k$. Observe that the set $\mathcal{S}(x)$ has exactly $\binom{n}{\nu_1(x),\ldots,\nu_k(x)}$ elements, where 
$$
\binom{n}{\lambda_1,\ldots,\lambda_k}=\frac{n!}{\prod_{j=1}^k \prod_{s \in S} \lambda_j\{s\}!}
$$
whenever $\lambda_1,\ldots,\lambda_k$ are point measures on $S$ such that $\lambda_j(S)=n_j$ for each $j=1,\ldots,k$. In other words, $\bm{U}_x$ represents the uniform distribution on $\mathcal{S}(x)$. 

Hereafter, the set of these point measures $\lambda=(\lambda_1,\ldots,\lambda_k)$ will be denoted by $\mathscr{L}$. In particular, for each $\lambda \in \mathscr{L}$, the product measure 
$$
\bigtimes_{j=1}^k\left(\frac{\lambda_j}{n_j}\right)^{n_j}
$$
is a probability measure on $(S^n,\mathscr{T}^n)$. Given $x\in S^n$, denote also by $\mathscr{V}_x$ the subset of point measures $\nu=(\nu_1,\ldots,\nu_k)$ in $\mathscr{L}$ supported on $\{x_1,\ldots,x_n\}$. Lastly, notice that the collection $\{\mathcal{S}(x): x \in S^n\}$ is made of pairwise disjoint subsets in $\mathscr{T}^n$ and their union is $S^n$. 

With these premises, it follows that, for each $x \in S^n$ and $\lambda \in \mathscr{L}$, it holds
$$
(\bigtimes_{j=1}^k \lambda_j^{n_j})(\mathcal{S}(x))=|\mathcal{S}(x)|\,\lambda^{\nu(x)},
$$
where $\alpha^\beta:=\prod_{j=1}^k\prod_{s \in S}\alpha_j\{s\}^{\beta_j\{s\}}$ for each $\alpha,\beta \in \mathscr{L}$. Hence, for each $x \in S^n$, $\lambda \in \mathscr{V}_x$, and $B \in \mathscr{T}^n$, we obtain
$$
(\bigtimes_{j=1}^k \lambda_j^{n_j})(B)=\sum_{\nu \in \mathscr{V}_x}\binom{n}{\nu_1,\ldots,\nu_k}\lambda^\nu \bm{U}_{\eta(\nu)}(B),
$$
with $\eta(\nu)$ being any element of $S^n$ for which $\nu(\eta(\nu))=\nu$. In particular, it represents a linear system of identities between measures on $(S^n,\mathscr{T}^n)$. Having fixed an order on the finite set $\mathscr{V}_x=\{\gamma_1,\ldots,\gamma_m\}$, the $m\times m$ square symmetric matrix of coefficients with $(i,j)$-th entry given by $\binom{n}{\gamma_i}\gamma_i^{\gamma_j}$ can be seen a leading principal minor of the bigger matrix obtained similarly by setting $k=1$. 

It has been shown in \cite[Theorem 2]{Moak} that all eigenvalues of the latter bigger matrix are (real, of course, and) positive, which is well known to be equivalent of being positive definite. In turn, all leading principal minors are positive definite by Sylvester's criterion. This implies that the above matrix of coefficients has full rank. Therefore, for each $x \in S^n$, there exist real coefficients $\{c_{x,\lambda}, \lambda \in \mathscr{V}_x\}$ such that for all $A \in \mathscr{S}^n$ it holds
$$
\bm{U}_x(A)=\sum_{\lambda \in \mathscr{V}_x}c_{x,\lambda} (\bigtimes_{j=1}^k \lambda_j^{n_j})(A).
$$

At this point, define the bounded signed measure $\mu:\sigma(\mathcal{P}(S))^k \to \mathbf{R}$ on the $k$-fold product $\mathcal{P}(S)^k$ by
$$
Q \mapsto \int_\Omega \sum_{\nu \in \mathscr{V}_{X(\omega)}}n_1^{n_1}\cdots n_k^{n_k} c_{X(\omega),\nu}(\bigtimes_{j=1}^k \delta_{\nicefrac{\nu_j}{n_j}})(Q)\,\mathrm{P}(\mathrm{d}\omega).
$$

Hence we obtain that
\begin{displaymath}
\begin{split}
\mathrm{P}((X_i: i \in I_j) \in A_j, j=1,\ldots,k)&=\int_\Omega \bm{U}_{X(\omega)}(A_1\times \cdots \times A_k)\,\mathrm{P}(\mathrm{d}\omega)\\
&=\int_\Omega \sum_{\lambda \in \mathscr{V}_{X(\omega)}}c_{X(\omega),\lambda} (\bigtimes_{j=1}^k \lambda_j^{n_j})(A_1\times \cdots \times A_k)\,\mathrm{P}(\mathrm{d}\omega)\\
&=\int_\Omega \sum_{\lambda \in \mathscr{V}_{X(\omega)}}n_1^{n_1}\cdots n_k^{n_k}c_{X(\omega),\lambda} \bigtimes_{j=1}^k (\frac{\lambda_j}{n_j})^{n_j}(A_1\times \cdots \times A_k)\,\mathrm{P}(\mathrm{d}\omega),
\end{split}
\end{displaymath}
which can be rewritten as
$$
\int_\Omega \sum_{\lambda \in \mathscr{V}_{X(\omega)}}n_1^{n_1}\cdots n_k^{n_k}c_{X(\omega),\lambda} \left(\int_{\mathcal{P}(S)^k}\nu_1^{n_1}\cdots \nu_k^{n_k}(\bigtimes_{j=1}^k \delta_{\nicefrac{\lambda_j}{n_j}})(\mathrm{d}\nu_1,\ldots,\mathrm{d}\nu_k)\right)(A_1\times \cdots \times A_k)\,\mathrm{P}(\mathrm{d}\omega)
$$
or, equivalently
$$
\int_\Omega \int_{\mathcal{P}(S)^k}\nu_1^{n_1}\cdots \nu_k^{n_k} \sum_{\lambda \in \mathscr{V}_{X(\omega)}}n_1^{n_1}\cdots n_k^{n_k}c_{X(\omega),\lambda} (\bigtimes_{j=1}^k \delta_{\nicefrac{\lambda_j}{n_j}})(\mathrm{d}\nu_1,\ldots,\mathrm{d}\nu_k)(A_1\times \cdots \times A_k)\,\mathrm{P}(\mathrm{d}\omega).
$$

This is equal by Fubini's theorem to
$$
\int_{\mathcal{P}(S)^k} \nu_1^{n_1}(A_1)\cdots \nu_k^{n_k}(A_k) \left(\int_\Omega\sum_{\lambda \in \mathscr{V}_{X(\omega)}}n_1^{n_1}\cdots n_k^{n_k}c_{X(\omega),\lambda} (\bigtimes_{j=1}^k \delta_{\nicefrac{\lambda_j}{n_j}})\mathrm{P}(\mathrm{d}\omega)\right) (\mathrm{d}\nu_1,\ldots,\mathrm{d}\nu_k),
$$
hence
$$
\mathrm{P}((X_i: i \in I_j) \in A_j, j=1,\ldots,k)=\int_{\mathcal{P}(S)^k} \nu_1^{n_1}(A_1)\cdots \nu_k^{n_k}(A_k)\,\mu(\mathrm{d}\nu_1,\ldots,\mathrm{d}\nu_k).
$$
\end{proof}

\begin{proof}[Proof of Corollary \ref{cor:parametric2}]
More generally, let us assume that $S$ is a locally convex topological vector space with induced Borel $\sigma$-field $\mathscr{S}$. Then there exists a bounded signed measure $\eta$ on $S^n$ such that
\begin{equation}\label{eq:representationfiniteexchangeableparametric}
\mathrm{P}\left((X_i: i \in I_j) \in A_j, j=1,\ldots,k\right)=\int_{S^n} \nu_{(\theta_i: i \in I_1)}^{n_1}(A_1) \cdots \nu_{(\theta_i: i \in I_k)}^{n_k}(A_k) \,\eta(\mathrm{d}\theta_1,\ldots,\mathrm{d}\theta_n)
\end{equation}
for all $A_1 \in \mathscr{S}^{n_1},\ldots,A_k \in \mathscr{S}^{n_k}$, where each $\nu_{(\theta_i: i \in I_j)}$ is a probability measure on $S$ depending measurably on $n_j$ parameters $(\theta_i: i \in I_j)$.

This is a straighforward generalization of \cite[Theorem 2]{KonstYuan0}, to which we refer the reader for the proof, taking in account also \cite[Theorem 3.4]{RudinFA}. In particular, since any two uncountable Polish spaces are isomorphic in the category of measurable spaces, it follows by a change of variable argument that the number of indexes reduces to one for each class $I_j$.
\end{proof}

\section{Proofs of Theorem \ref{th:truemixtures} and \ref{th:positive}}\label{sec:truemixtures}

\begin{proof}[Proof of Theorem \ref{th:truemixtures}] According to the standing assumptions, there exists a probability measure $\mu$ on $\mathcal{P}(S)$ such that, for all nonnegative integers $p_j,q_j,r_j$ summing up to $n_j$, and for each $A_j \in \mathscr{S}$ and $B_j \in \mathscr{S}^{q_j}$, it holds
$$
\mathrm{P}((X_i)_{i \in I_j} \in A_j^{p_j} \times B_j \times S^{r_j}, j=1,\ldots,k)=\int_{\mathcal{P}(S)^k}\prod_{j=1}^k\nu_j(A_j)^{p_j} \underbrace{(\nu_j \times \cdots \times \nu_j)}_{q_j\text{ times}}(B_j)\, \mu(\mathrm{d}\nu_1,\ldots,\mathrm{d}\nu_k).
$$
The claim follows by Lemma \ref{lm:CBS}, where each function $f_j$ is given by $\mathcal{P}(S) \to \mathbf{R}: \nu \mapsto \nu(A_j)$ and the function $g_j$ by $\mathcal{P}(S) \to \mathbf{R}: \nu \mapsto \nu^{q_j}(B_j)$.
\end{proof}

\begin{proof}[Proof of Theorem \ref{th:positive}] By the exchangeability assumption, it holds $\sum_{i=0}^n \binom{n}{i}x_i=1$. Hence, we are asking for necessary and sufficient conditions on the (column) vector $x=(x_0,\ldots,x_n)^\prime$ such that there exists a probability measure $\mu:\mathscr{B}[0,1] \to \mathbf{R}$ for which
$
x_i=\int_{[0,1]}p^{n-i}(1-p)^{i} \mu(\mathrm{d}p)
$
 for all $i=0,1,\ldots,n$. Moreover, let us define $y_i$ as the $(n-i)$-th moment of $\mu$. 
It follows that the above conditions are equivalent to
\begin{equation}\label{eq:moments2}
x_i=\sum_{j=0}^i \binom{i}{j}(-1)^{i+j}y_j
\end{equation}
for each $i=0,1,\ldots,n$, where by convention $\binom{0}{0}=1$. In addition, define $y=(y_n,\ldots,y_0)^\prime$ and let $A$ be the square matrix of dimension $(n+1)$ where the $(i,j)$-th element is 
$$
\binom{i-1}{n+2-i-j}(-1)^{i+j+n}\bm{1}_{[n+2,\infty)}(i+j).
$$
Accordingly, the system \eqref{eq:moments2} can be rewritten in matrix form as $x=Ay$. Since the determinant of $A$ is $1$, each $y_i$ can be rewritten uniquely as a linear combination of the $x_j$.  

Let us prove by induction that $y_i=\sum_{j=0}^i \binom{i}{j}x_j$ for each $i=0,1,\ldots,n$. It is trivially true for $n=0$, and let us suppose that it holds for all nonnegative integers smaller than $n$. Then, by the system \eqref{eq:moments2} and the inductive hypothesis, we obtain
\begin{displaymath}
\begin{split}
y_n&=x_n-\sum_{i=0}^{n-1}\binom{n}{i}(-1)^{n+i}y_i=x_n-\sum_{i=0}^{n-1}\binom{n}{i}(-1)^{n+i}\sum_{j=0}^i \binom{i}{j}x_j\\
&=x_n-(-1)^{n}\sum_{i=0}^{n-1}\binom{n}{i}(-1)^{i}\sum_{j=0}^i \binom{i}{j}x_j=x_n-(-1)^{n}\sum_{j=0}^{n-1}x_j\sum_{i=j}^{n-1}\binom{n}{i}\binom{i}{j}(-1)^i.
\end{split}
\end{displaymath}
Hence, it follows by Lemma \ref{lm:identity} that
\begin{equation}\label{eq:xnfinal}
y_n=x_n-(-1)^{n}\sum_{j=0}^{n-1}\binom{n}{j}(-1)^{n-1}x_j=\sum_{j=0}^n \binom{n}{j}x_j.
\end{equation}

Having this in mind, the joint law of $(X_1,\ldots,X_n)$ is a true mixture of i.i.d. $\{0,1\}$-valued random variables if and only if there exists a probability measure $\mu:\mathscr{B}[0,1]\to \mathbf{R}$ which satisfies
$$
\int_{[0,1]} p^i \mu(\mathrm{d}p)=y_{n-i}
$$
for each $i=0,1,\ldots,n$. This problem, commonly known as ``reduced Hausdorff moment problem,'' have been extensively studied in literature; see, for instance, \cite[pp. 8-9]{Shohat} and \cite{Hilde}. According to \cite[Theorem 1]{Jurkat}, a necessary and sufficient condition for the existence of a solution is that, if $n$ is even, both matrices
$$
\begin{pmatrix}
  y_{n} & y_{n-1} & \cdots & y_{\frac{n}{2}} \\
  y_{n-1} & y_{n-2} & \cdots & y_{\frac{n}{2}-1} \\
  \vdots  & \vdots  & \ddots & \vdots  \\
  y_{\frac{n}{2}} & y_{\frac{n}{2}-1} & \cdots & y_0
 \end{pmatrix}
\,\,\,\,\text{and}\,\,\,\,  
\begin{pmatrix}
  y_{n-1}-y_{n-2} & y_{n-2}-y_{n-3} & \cdots & y_{\frac{n}{2}}-y_{\frac{n}{2}-1} \\
  y_{n-2}-y_{n-3} & y_{n-3}-y_{n-4} & \cdots & y_{\frac{n}{2}-1}-y_{\frac{n}{2}-2} \\
  \vdots  & \vdots  & \ddots & \vdots  \\
  y_{\frac{n}{2}}-y_{\frac{n}{2}-1} & y_{\frac{n}{2}-1}-y_{\frac{n}{2}-2} & \cdots & y_1-y_0
 \end{pmatrix}
$$
are positive semi-definite, whereas in case $n$ is odd, both matrices
$$
\begin{pmatrix}
  y_{n-1} & y_{n-2} & \cdots & y_{\frac{n-1}{2}} \\
  y_{n-2} & y_{n-3} & \cdots & y_{\frac{n-3}{2}} \\
  \vdots  & \vdots  & \ddots & \vdots  \\
  y_{\frac{n-1}{2}} & y_{\frac{n-3}{2}} & \cdots & y_0
 \end{pmatrix}
\,\,\,\,\text{and}\,\,\,\,  
\begin{pmatrix}
  y_{n}-y_{n-1} & y_{n-1}-y_{n-2} & \cdots & y_{\frac{n+1}{2}}-y_{\frac{n-1}{2}-1} \\
  y_{n-1}-y_{n-2} & y_{n-2}-y_{n-3} & \cdots & y_{\frac{n-1}{2}}-y_{\frac{n-3}{2}} \\
  \vdots  & \vdots  & \ddots & \vdots  \\
  y_{\frac{n+1}{2}}-y_{\frac{n-1}{2}} & y_{\frac{n-1}{2}}-y_{\frac{n-3}{2}} & \cdots & y_1-y_0
 \end{pmatrix}
$$
are positive semi-definite too. The claim follows by equation \eqref{eq:xnfinal} and standard properties of the binomial coefficients.
\end{proof}

\begin{proof}[Proof of Corollary \ref{cor:reinfor}]
For each $i=1,\ldots,2\lfloor n/2\rfloor-1$, the inequality \eqref{eq:reinf} follows by Lemma \ref{lm:CBS}. Indeed, having fixed $k=1$, $m_1=\lfloor n/2\rfloor$, $A=\{1\}$, and $B=\{0,1\}$ if $n$ is odd, we obtain that the matrix
$$
\begin{pmatrix}
  x_{2m} & x_{2m-1} & \cdots & x_{m} \\
  x_{2m-1} & x_{2m-2} & \cdots & x_{m-1} \\
  \vdots  & \vdots  & \ddots & \vdots  \\
  x_{m} & x_{m-1} & \cdots & x_0
 \end{pmatrix}
$$
is positive semi-definite. In particular, each principal minor $\begin{pmatrix}
  x_{i+1} & x_{i} \\
  x_{i} &  x_{i-1}
 \end{pmatrix}
$ has nonnegative determinant whenever $i+1 \le 2\lfloor n/2\rfloor$. 

Therefore, we miss only to prove that, if $n$ is odd, then $x_{n-1} \le \sqrt{x_nx_{n-2}}$. This is easily seen, according to Theorem \ref{th:positive}, by choosing the $2\times 2$ leading principal minor of $\textsl{K}_n$, which has to have nonnegative determinant. Indeed, by the standard properties of the determinant, we conclude that
$$
\mathrm{det}\begin{pmatrix}
  y_{n}-y_{n-1} & y_{n-1}-y_{n-2} \\
  y_{n-1}-y_{n-2} &  y_{n-2}-y_{n-3}
 \end{pmatrix}
=
\mathrm{det}\begin{pmatrix}
  x_n & x_{n-1} \\
  x_{n-1} &  x_{n-2}
 \end{pmatrix}
\ge 0.
$$

On the other hand, for $n\le 3$, the sufficiency of these inequalities, follows by Theorem \ref{th:positive}. Conversely, again by Theorem \ref{th:positive}, given an integer $n\ge 4$, it is enough to show the existence of a sequence $(x_0,\ldots,x_n)$ such that the inequalities $x_i\le  \sqrt{x_{i-1}x_{i+1}}$ hold for all $i=1,\ldots,n-1$, while the matrix $\textsl{H}_n$ admits a principal minor with negative determinant. To this aim, set 
$$
(x_0,\ldots,x_n)=\frac{1}{3\cdot 2^n+2n+6}\,(9,5,\underbrace{3,3,\ldots,3}_{n-1\text{ times}}).
$$
Then, it is routine to check that this defines an exchangeable law on $\{0,1\}^n$, while, on the other hand, the $3\times 3$ south-east principal minor has a negative determinant.
\end{proof}

By Sylvester's criterion, a matrix is positive semi-definite if and only if its principal minors have nonnegative determinant. Therefore, Theorem \ref{th:positive} allows to establish whether each $\{0,1\}$-valued exchangeable sequence $(X_1,\ldots,X_n)$ is a true mixture of i.i.d. laws in $o(2^n)$ operations (indeed the number of principal minors of a matrix of dimension $n$ is $\sum_{i=1}^n \binom{n}{i}^2$, which is asymptotically equal to $\frac{4^n}{\sqrt{\pi n}}$). Just to realize the computational burden related to Theorem \ref{th:positive}, for $n=4$, notice that the exchangeable law defined by the sequence $(x_0,\ldots,x_4)$ is a true mixture of $\{0,1\}$-valued i.i.d. random variables if and only if each of the following numbers are nonnegative:
\begin{enumerate}[label={\rm (\roman{*})}]
\item\label{item1}  $x_0x_2-x_1^2$; 
\item\label{item2}  $x_1x_3-x_2^2$; 
\item\label{item3}  $2x_0x_2- 2x_1^2 - x_2x_1 + x_0x_3$;  
\item\label{item4}  $x_0x_3 - x_1^2 - x_1x_2 + x_3x_1 - x_2^2 + x_0x_2$; 
\item\label{item5}  $x_0x_4x_2 - x_4x_1^2 + 2x_1x_2x_3 - x_2^3 - x_0x_3^2$; 
\item\label{item6}  $4x_0x_3 - 4x_1^2 - 4x_1x_2 - x_2^2 + 4x_0x_2 + x_0x_4$; 
\item\label{item7}  $2x_0x_2 + 3x_0x_3 - 3x_1x_2 + x_0x_4 + 2x_1x_3 + x_1x_4 - x_2x_3 - 2x_1^2 - 3x_2^2$; 
\item\label{item8}  $x_0x_2 + 2x_0x_3 - 2x_1x_2 + x_0x_4 + 2x_1x_3 + 2x_1x_4 - 2x_2x_3 + x_2x_4 - x_1^2 - 3x_2^2 - x_3^2$. 
\end{enumerate}

Nevertheless, it would be sufficient that the \emph{leading} principal minors of $\textsl{H}_n$ and $\textsl{K}_n$ have positive determinants (in the above example, e.g., it means that $x_0,\ldots,x_4$, and numbers \ref{item1}, \ref{item2}, and \ref{item7} are strictly positive). It is unclear whether this result could be extended to arbitrary spaces $S$.



\section{Open questions}\label{sec:openq}

Generally, given a bounded signed measure $\mu$ on $\mathcal{P}(S)$ such that $\mu(\mathcal{P}(S))=1$, it is not true that the map
$$
\mathscr{S}^n \to \mathbf{R}: A\mapsto \int_{\mathcal{P}(S)}\nu^n(A) \mu(\mathrm{d}\nu)
$$
represents the joint law of some exchangeable sequence $(X_1,\ldots,X_n)$ with values in $S$. Indeed, the integral may attain negative values. Accordingly, can we provide a characterization of the set $\mathcal{M}$ of such signed measures $\mu$? Is there a way to represent $\mu$ itself?

In addition, with the notations of Theorem \ref{th:finiterepresentation}, is it true that 
$$
\frac{n_1+\cdots+n_k}{k} \to \infty \,\,\,\,\,\text{ implies }\,\,\,\,\,\mu^-(\mathcal{P}(S)^k)\to 0\,\,?
$$

Finally, given subgroups $G$ and $G^\prime$ such that $G \subseteq G^\prime$, can we quantify how much $\mathcal{M}(G)$ is ``larger'' than $\mathcal{M}(G^\prime)$?


\section*{Acknowledgments}
The author is supported by a PhD scholarship from Universit\`a Bocconi. He thanks Nate Eldredge (Northern Colorado, US) and Pierpaolo Battigalli, Sandra Fortini, Fabio Maccheroni, Pietro Muliere, and Sonia Petrone (Universit\`a Bocconi, IT) for useful comments.



\end{document}